\documentclass[oneside]{amsart}

\usepackage[letterpaper,body={14.0cm,22.0cm}, mag=1000]{geometry}
\usepackage{amssymb}
\usepackage{amsthm}
\usepackage{amscd}
\usepackage{enumitem}
\usepackage{float}
\usepackage{placeins}
\usepackage{caption}

\numberwithin{equation}{section}
\theoremstyle{plain}
\newtheorem{alg}[equation]{Algorithm}

\newtheorem{lemma}[equation]{Lemma}

\newtheorem{prop}[equation]{Proposition}

\newtheorem{thm}[equation]{Theorem}

\newtheorem{conj}[equation]{Conjecture}

\theoremstyle{definition}

\newcommand{\dlabel}[1]{\ifmmode \text{\ttfamily \upshape [#1] } \else
{\ttfamily \upshape [#1] }\fi \label{#1}}

\newcommand{\Aut}{\operatorname{Aut} }

\newcommand{\Hol}{\operatorname{Hol} }
\newcommand{\Symm}{\operatorname{Symm} }
\newcommand{\Reg}{\operatorname{Reg} }

\begin{document}
\setcounter{page}{1}
\title[Computing skew left braces of small orders]
{Computing skew left braces of small orders}

\author{Valeriy G. Bardakov}
\address{Sobolev Institute of Mathematics, pr. ak. Koptyuga 4, Novosibirsk, 630090, Russia  and Novosibirsk State University, Novosibirsk, 630090, Russia}
\email{bardakov@math.nsc.ru}

\author{Mikhail V. Neshchadim}
\address{Sobolev Institute of Mathematics, pr. ak. Koptyuga 4, Novosibirsk, 630090, Russia and Novosibirsk State University, Novosibirsk, 630090, Russia}
\email{neshch@math.nsc.ru}

\author{Manoj K. Yadav}
\address{Harish-Chandra Research Institute, HBNI, Chhatnag Road, Jhunsi, Allahabad-211 019, India}
\email{myadav@hri.res.in}

\subjclass[2010]{16T25, 81R50, 20B40}
\keywords{Skew left braces, Left braces, Regular subgroups, Yang-Baxter equation}

\begin{abstract}
We improve  Algorithm 5.1 of [Math. Comp. {\bf 86} (2017), 2519-2534] for computing all non-isomorphic skew left braces, and   enumerate left braces and skew left braces of  orders up to 868 with some exceptions. Using the enumerated data, we state some conjectures for further research.
\end{abstract}
\maketitle

\section{Introduction}
 
 A  triple $(G, +, \circ)$, where $(G, +)$ and $(G, \circ)$ are (not necessarily abelian) groups,    is said to be a \emph{skew left brace} if 
 \begin{equation}
 g_1 \circ (g_2 + g_3) =  (g_1 \circ g_2) - g_1 + (g_1 \circ g_3)
 \end{equation}
 for all $g_1, g_2, g_3 \in G$, where $- g_1$ denotes the  inverse of $g_1$ in $(G, +)$. We call  $(G, +)$ the \emph{additive group} and $(G, \circ)$ the \emph{multiplicative  group} of the skew left brace $(G, +, \circ)$. A skew left brace 
 $(G, +, \circ)$ is said to be a \emph{left brace} if $(G, +)$ is an abelian group. The concept of left braces was introduced by Rump \cite{R2007} in 2007 in connection with non-degenerate involutive set theoretic solutions of the quantum Yang-Baxter equations. Thereafter the subject received a tremendous attention of the mathematical community; see \cite{BCJO18, FC2018, WR2019, AS2018} and the references therein.  Interest in the study of set theoretic solutions of the quantum Yang-Baxter equations was intrigued by the paper \cite{D1992}  of Drinfeld, published in 1992.

 Let $X$ be an arbitrary set and $R : X \times X \to X \times X$  a bijective map. Recall that the pair  $(X, R)$ is said to be  a set theoretic solution of the Yang-Baxter equation if 
 $$R_{12}R_{23}R_{12} = R_{23}R_{12}R_{23}$$
holds in the set of all maps from $X \times X \times X$ to itself, where $R_{ij}$ is just $R$ acting on the $i$th and $j$th components of $X \times X \times X$ and identity on the
remaining one. Let us write 
$$R(x, y) = \big(\sigma_x(y), \tau_y(x)\big), ~ x, y \in X$$
with $\sigma_x$ and $\tau_y$ component  maps from  $X$ to itself. 

A solution $(X, R)$ is said to be non-degenerate if the component maps $\sigma_x$ and $\tau_y$  are bijections on $X$ for all $x, y \in X$. It is said to be involutive if $R^2$ is the identity map.  The study of non-degenerate set theoretic solutions of the  quantum Yang-Baxter equations has been extensively taken up, e. g.,  \cite{CJR10, PD2015, ESS99, GI2018, LV2016} to mention a few.

 The concept of skew left brace was introduced by Guarnieri and Vendramin \cite{GV2017} in 2017 in connection with non-involutive non-degenerate set theoretic solutions of the quantum Yang-Baxter equations. They invented an algorithm, by generalising a result of Bachiller \cite{DB2016} for computing all skew left braces of a given order.  They themselves computed left braces and skew left braces of lot of groups upto order 120. Vendramin \cite{LV2019} extended the number upto 168 with some exceptions. All these computations are done using computer algebra systems MAGMA \cite{magma} and GAP \cite{GAP} using the algorithm invented in \cite{GV2017}. For more work on skew braces see \cite{CSV19, LC2018, KN2019}.

 This article aims at filling up the gaps in the table produced in \cite{LV2019} to some extent and making further computations for larger orders.  An ingenious observation on regular subgroups of the holomorph of a given finite group allows us to improve the algorithm obtained in \cite{GV2017}, which substantially  enhances the performance of MAGMA computation. The improved algorithm, actually, avoids an expensive calculation in the existing algorithm. We compute the number of non-isomorphic left braces and skew left braces of orders upto 868 except certain cases (mainly when the order is a multiple of 32). These results settle \cite[Problem 13]{LV2019}  and \cite[Problem 6.1]{GV2017}.  The computations will help in building a database of left braces and skew left braces, which in turn will greatly enrich the library of solutions of the quantum Yang-Baxter equation. On the basis of our computation, we suggest some conjectures for further research.
 
  It is  striking that there are more than a million  skew brace structures of order $2^5$ and more than 20 millions  skew brace structures of order $3^5$. The reader will encounter many more surprises  while going through the tables. We have used MAGMA on a computer with 3.5 GHz 6-Core Intel Xeon E5 processor and 64 GB memory for these computations.

\section{Regular subgroups}

Let $G$ be a group, which acts on a set $X$. The action of an element $g \in G$ on an element $x \in X$ is denoted by $x^g$. A subgroup $H$ of $G$ is said to be \textit{action-closed} if for each pair $(g, x) \in G \times X$, there exists an element $h \in H$ such that $x^g = x^h$. By  \textit{$H$ - conjugacy class} of $x \in G$, we mean $\{x^h \mid h \in H\}$. For $g, h \in G$, we write the conjugate of $g$ by $h$ as $g^h = h^{-1}gh$.

Let $G$ be a  group and $\Symm(G)$ be the symmetric group on the set $G$. Recall that a subgroup  $\mathcal{G}$ of $\Symm(G)$ is said to be \textit{regular} if $\mathcal{G}$-action on $G$ is free and  transitive. By a free action we here mean that for any element $g \in G$, its stabilizer in $\mathcal{G}$ is the trivial subgroup. Observe that when $G$ is finite, any regular subgroup of $\Symm(G)$ is of order $|G|$. 

For a group $G$, $\Hol(G)$ denotes the holomorph of $G$, which is defined as the semidirect product of $G$ with $\Aut(G)$, the automorphism group of $G$. So
\[\Hol(G) := \Aut(G)   \ltimes G,\]
where the product in $\Hol(G)$ is given by
\[(\alpha, g)(\beta, h) = (\alpha \beta, g\alpha(h)).\]
Notice that $\Hol(G)$ acts on $G$ transitively under the following action:
\[g^{(\alpha, h)} = \pi_2\big((\alpha, h) (1, g)\big) = h\alpha(g)\]
for all $\alpha \in \Aut(G)$ and $g, h \in G$, where $\pi_2 : \Hol(G) \to G$ is the projection map given by $\pi_2\big( (\alpha, g)\big) = g$.   It follows that the stabilizer of any element of $G$ in $\Hol(G)$ is isomorphic to $\Aut(G)$. 

Let $\mathcal{G}$ be a regular subgroup of $\Hol(G)$. Then it is not difficult to see that for each $g \in G$, there exists a unique element $(\alpha, h) \in \mathcal{G}$ such that $g^{(\alpha, h)} = h\alpha(g) = 1$. Let $\Reg(G)$ denote the set of all regular subgroups of $\Hol(G)$. Then $\Hol(G)$ acts on $\Reg(G)$ by conjugation.  With this setting, we have
the following easy observation, which plays a key role in what follows.

\begin{lemma}\label{key-lemma}
$\Aut(G)$, as a subgroup of $\Hol(G)$, is action-closed  with respect to the conjugation action of $\Hol(G)$ on $\Reg(G)$.
\end{lemma}
\begin{proof}
 Let $\mathcal{G} \in \Reg(G)$ and $(\alpha, h) \in \Hol(G)$. Then there exists an element $(\alpha_1, h_1) \in \mathcal{G}$ such that $h^{(\alpha_1, h_1)} = h_1\alpha_1(h) = 1$. Notice  that 
$$(\alpha_1, h_1)(\alpha, h) = \big(\alpha_1 \alpha,  h_1\alpha_1(h)\big) = (\alpha_1 \alpha, 1).$$
Let $\beta := \alpha_1 \alpha$, which lies in $\Aut(G)$. Thus,
\[\mathcal{G}^{(\beta,1)} = \mathcal{G}^{(\alpha_1, h_1)(\alpha, h)} = \big(\mathcal{G}^{(\alpha_1, h_1)}\big)^{(\alpha, h)} = \mathcal{G}^{(\alpha, h)}.\]
Proof is now complete.
\end{proof}

The preceding lemma enables us to get the following generalization of  \cite[Proposition 4.3]{GV2017}.

\begin{thm}\label{thm}
Let $(G, +)$ be a group. Then non-isomorphic skew left braces  $(G, +, \circ)$ are in bijective correspondence with conjugacy classes of regular subgroups in $\Hol(G, +)$. Moreover, if $G$ is a $p$-group for some prime $p$, then non-isomorphic skew left brace structures over $G$ are in bijective correspondence with  $\Aut(G)$ - conjugacy classes of regular subgroups of any Sylow $p$-subgroup of $\Hol(G, +)$.
\end{thm}
\begin{proof}
The first assertion  follows from \cite[Proposition 4.3]{GV2017} along with Lemma \ref{key-lemma}.  Let $\mathcal{S}$ be a fixed Sylow $p$-subgroup of $\Hol(G, +)$ and $\mathcal{S}'$ any other Sylow $p$-subgroup of $\Hol(G, +)$. In the light of first assertion,  we only need to observe that  any regular subgroup of  $\mathcal{S}'$ lies in the $\Hol(G, +)$ - conjugacy class of some regular subgroup of $\mathcal{S}$. But this is obvious by Sylow theory.
\end{proof}

As a result, we get the following  algorithm which improves \cite[Algorithm 5.1]{GV2017}.

\begin{alg}\label{alg1}
For a finite group $(G, +)$, the following sequence of computations constructs all non-isomorphic skew left braces  $(G, +, \circ):$
\begin{enumerate}
\item Compute  $\Hol(G, +)$.
\item Compute the list of regular subgroups of $\Hol(G)$ of order $|G|$ up to conjugation.
\item For each representative $\mathcal{G}$ of regular subgroups of $\Hol(G)$, construct the map $\chi : G \to \mathcal{G}$ given by $g \mapsto \big(f, f(g)^{-1}\big)$, where $ \big(f, f(g)^{-1}\big) \in \mathcal{G}$. The triple $(\mathcal{G}, G, \chi)$ yields a skew left brace $(G, +, \circ)$ with $\circ: G \times G \to G$ given by $g_1 \circ g_2 = \chi^{-1}\big(\chi(g_1)\chi(g_2)\big)$ for all $g_1, g_2 \in G$.
\end{enumerate}
\end{alg}

As remarked in \cite{GV2017} too, for enumerating skew left braces with additive group $(G, +)$ we only need first two steps of this algorithm.

We also have the following algorithm for finite $p$-groups.

\begin{alg}\label{alg2}
For a finite $p$-group $(G, +)$, the following sequence of computations constructs all non-isomorphic skew left braces  $(G, +, \circ):$
\begin{enumerate}
\item Compute  $\Hol(G, +)$.
\item Compute a representative $\mathcal{S}$  of the conjugacy class of Sylow $p$-subgroups of  $\Hol(G, +)$.
\item Compute the list of regular subgroups of $\mathcal{S}$ of order $|G|$ up to conjugation by the elements of $\Aut(G)$.
\item For each representative $\mathcal{G}$ of regular subgroups of $\mathcal{S}$ under conjugation action of $\Aut(G)$, construct the map $\chi : G \to \mathcal{G}$ given by $g \mapsto \big(f, f(g)^{-1}\big)$, where $ \big(f, f(g)^{-1}\big) \in \mathcal{G}$. The triple $(\mathcal{G}, G, \chi)$ yields a skew left brace $(G, +, \circ)$ with $\circ: G \times G \to G$ given by $g_1 \circ g_2 = \chi^{-1}\big(\chi(g_1)\chi(g_2)\big)$ for all $g_1, g_2 \in G$.
\end{enumerate}
\end{alg}

 Notice that for enumerating skew left braces with finite additive $p$-group $(G, +)$ we  need first three steps of this algorithm. We conclude this section by reproducing the proof of the following fact.
 
 \begin{prop}
 Let $\mathcal{S}$ be a Sylow $p$-subgroup of $\Hol(G)$ of a finite $p$-group $G$. Then the union of $\Aut(G)$ - conjugacy classes of the regular subgroups of $\mathcal{S}$ constitutes the set of all regular subgroups of $\Hol(G)$.
 \end{prop}
 \begin{proof}
 Let $\mathcal{G}$  be an arbitrary regular subgroup of $\Hol(G)$. Then $\mathcal{G}$, being of order $|G|$, is a subgroup of some Sylow $p$-subgroup $\mathcal{S}'$ of  $\Hol(G)$. By Sylow theory, we know that there exists an element $(\phi, y) \in \Hol(G)$ such that $\mathcal{S} = (\mathcal{S}')^{(\phi,y)}$. Thus $\mathcal{G}^{(\phi,y)}$ is a subgroup of $\mathcal{S}$. It follows from the proof of Lemma \ref{key-lemma} that   $\mathcal{G}^{(\phi,y)} = \mathcal{G}^{(\beta, 1)}$ for some $\beta \in \Aut(G)$. A routine calculation now shows that  $\mathcal{G}^{(\beta, 1)}$ is a regular subgroup of  $\mathcal{S}$. Indeed, if $x^{\big((\psi, z)^{(\beta,1)}\big)} = x$ for some $(\psi, z) \in \mathcal{G}$ and $x \in G$, then it follows that $\beta(x)^{(\psi, z)} = \beta(x)$, which is not possible. This proves that the action of  $\mathcal{G}^{(\beta, 1)}$ is free on $G$.  That the action is transitive, is left as an easy exercise,  and the proof is complete.
\end{proof}
 
 We remark that  on the lines of  proof of the preceding proposition, we can easily show that for an arbitrary finite group $G$, $\Hol(G)$ acts  on $\Reg(G)$ by conjugation. We have used this fact above without proof as it is well known.


\section{Computations}

Throughout this section, for a given positive integer $n$, $b(n)$ and $s(n)$, respectively, denote the total number of left braces and skew left braces of order $n$.  For each such $n$, $pf(n)$ stands for the prime factorization of $n$. Computations in this section are carried out  using Algorithm \ref{alg1}. The following table remedies some gaps in the list obtained in  \cite{LV2019}. 

\begin{table}[H]
\centering
\begin{small}
 \begin{tabular}{|c | c c c c c c c c|} 
 \hline
 $n$ & $32$ & $54$ & $64$ & $72$ &  $80$ & $81$ & $96$ & $108$ \\ 
 $b(n)$ & $25281$ & $80$ & $?$ & $489$ & $1985$ & $804$ & $195971$ & $494$ \\
 $s(n)$ & $1223061$ & $1028$ & $?$ & $17790$ & $74120$ & $8436$ & $?$ & $11223$ \\
 \hline
  $n$ & $112$ & $120$ & $126$ & $128$ &  $136$ & $144$ & $147$ & $150$\\ 
   $b(n)$ & $1671$ & $395$ & $36$ & $?$ & $108$ & $10215$ & $9$ & $19$\\
  $s(n)$ & $65485$ & $22711$ & $990$ & $?$ & $986$ & $3013486$ & $123$ & $401$\\
  \hline
 $n$  & $152$ & $158$ & $160$ & $162$ & $164$ & $165$ & $166$ & $168$\\
   $b(n)$  & $90$ & $2$ & $209513$ & $1374$ & $11$ & $2$ & $2$ & $443$\\
    $s(n)$  & $800$ & $6$ & $?$ & $45472$ & $43$ & $12$ & $6$ & $28505$\\
         \hline
\end{tabular}
\hspace{1cm}
\caption{Some missing values from \cite{LV2019}}\label{Table1}
\end{small}
\end{table}
    
 We now enumerate $b(n)$ and $s(n)$ for $n \le 868$ except some cases for which computations are too big to be handled by our computer.   We have given a lower bound on the number of skew left braces of order $3^5$, by taking into account the additive groups with Group Id's $[243, m]$, where $m = 1, \ldots,  31,  33,  37,  38,  40, 48, \ldots, 63, 65, 66,  67$. By the Group Id we mean the group identification of a group of given order in The Small Groups Library \cite{SmallGrp} implemented in GAP and MAGMA.
 
\begin{table}[H]
\centering
\begin{small}
 \begin{tabular}{|c | c c c c c c c c c c|} 
      \hline
 $n$ & $169$ & $170$ & $171$ & $172$ & $173$ & $174$ & $175$ & $176$ & $177$ & $178$\\
   $b(n)$ & $4$ & $4$ & $14$ & $9$ & $1$ & $4$ & $4$ & $1670$ & $1$ & $2$\\
    $s(n)$ & $4$ & $36$ & $80$ & $29$ & $1$ & $36$ & $4$ & $65466$ & $1$ & $6$\\
    \hline
 $n$ & $179$ & $180$ & $181$ & $182$ & $183$ & $184$ & $185$ & $186$ & $187$ & $188$\\
   $b(n)$ & $1$ & $129$ & $1$ & $4$ & $2$ & $90$ & $1$ & $6$ & $1$ & $9$\\
    $s(n)$ & $1$ & $5849$ & $1$ & $36$ & $8$ & $800$ & $1$ & $78$ & $1$ & $29$\\
    \hline
 $n$ & $189$ & $190$ & $191$ & $192$ & $193$ & $194$ & $195$ & $196$ & $197$ & $198$\\
   $b(n)$ & $165$ & $4$ & $1$ & $?$ & $1$ & $2$ & $2$ & $41$ & $1$ & $16$\\
    $s(n)$ & $4560$ & $36$ & $1$ & $?$ & $1$ & $6$ & $8$ & $389$ & $1$ & $294$\\
      \hline
 $n$ & $199$ & $200$ & $201$ & $202$ & $203$ & $204$ & $205$ & $206$ & $207$ & $208$\\
   $b(n)$ & $1$ & $568$ & $2$ & $2$ & $2$ & $28$ & $2$ & $2$ & $4$ & $1984$\\
    $s(n)$ & $1$ & $23471$ & $8$ & $6$ & $16$ & $410$ & $12$ & $6$ & $4$ & $74104$\\
      \hline
 $n$ & $209$ & $210$ & $211$ & $212$ & $213$ & $214$ & $215$ & $216$ & $217$ & $218$\\
   $b(n)$ & $1$ & $12$ & $1$ & $11$ & $1$ & $2$ & $1$ & $5308$ & $1$ & $2$\\
    $s(n)$ & $1$ & $468$ & $1$ & $43$ & $1$ & $6$ & $1$ & $523768$ & $1$ & $6$\\
      \hline
 $n$ & $219$ & $220$ & $221$ & $222$ & $223$ & $224$ & $225$ & $226$ & 227$$ & $228$\\
   $b(n)$ & $2$ & $36$ & $1$ & $6$ & $1$ & $195483$ & $21$ & $2$ & $1$ & $34$ \\
    $s(n)$ & $8$ & $702$ & $1$ & $78$ & $1$ & $?$ & $61$ & $6$ & $1$ & $606$\\
      \hline
 $n$ & $229$ & $230$ & $231$ & $232$ & $233$ & $234$ & $235$ & $256$ & $237$ & $238$\\
   $b(n)$ & $1$ & $4$ & $2$ & $106$ & $1$ & $36$ & $1$ & $9$ & $2$ & $4$\\
    $s(n)$ & $1$ & $36$ & $8$ & $944$ & $1$ & $990$ & $1$ & $29$ & $8$ & $36$\\
      \hline
 $n$ & $239$ & $240$ & $241$ & $242$ & $243$ & $244$ & $245$ & $246$ & $247$ & $248$\\
   $b(n)$ & $1$ & $10518$ & $1$ & $8$ & $598065$ & $11$ & $4$ & $4$ & $1$ & $90$\\
    $s(n)$ & $1$ & $4642485$ & $1$ & $57$ & $> 27447027$ & $43$ & $4$ & $36$ & $1$ & $800$\\
      \hline
 $n$ & $249$ & $250$ & $251$ & $252$ & $253$ & $254$ & $255$ & $256$ & $257$ & $258$\\
   $b(n)$ & $1$ & $104$ & $1$ & $229$ & $2$ & $2$ & $1$ & $?$ & $1$ & $6$\\
    $s(n)$ & $1$ & $1492$ & $1$ & $11541$ & $24$ & $6$ & $1$ & $?$ & $1$ & $78$\\
      \hline
 $n$ & $259$ & $260$ & $261$ & $262$ & $263$ & $264$ & $265$ & $266$ & $267$ & $268$\\
   $b(n)$ & $1$ & $35$ & $4$ & $2$ & $1$ & $345$ & $1$ & $4$ & $1$ & $9$\\
    $s(n)$ & $1$ & $739$ & $4$ & $6$ & $1$ & $20231$ & $1$ & $36$ & $1$ & $29$\\
      \hline
 $n$ & $269$ & $270$ & $271$ & $272$ & $273$ & $274$ & $275$ & $276$ & $277$ & $278$\\
   $b(n)$ & $1$ & $160$ & $1$ & $2014$ & $5$ & $2$ & $13$ & $24$ & $1$ & $2$\\
    $s(n)$ & $1$ & $6168$ & $1$ & $74960$ & $113$ & $6$ & $93$ & $324$ & $1$ & $6$\\
      \hline
 $n$ & $279$ & $280$ & $281$ & $282$ & $283$ & $284$ & $285$ & $286$ & $287$ & $288$\\
   $b(n)$ & $11$ & $385$ & $1$ & $4$ & $1$ & $9$ & $2$ & $4$ & $1$ & $1392959$\\
    $s(n)$ & $47$ & $22295$ & $1$ & $36$ & $1$ & $29$ & $8$ & $36$ & $1$ & $?$\\
    \hline
    
    $n$ & $289$ & $290$ & $291$ & $292$ & $293$ & $294$ & $295$ & $296$ & $297$ & $298$\\
   $b(n)$ & $4$ & $4$ & $2$ & $11$ & $1$ & $31$ & $1$ & $106$ & $37$ & $2$\\ 
  $s(n)$ & $4$ & $36$ & $8$ & $43$ & $1$ & $2152$ & $1$ & $944$ & $101$ & $6$\\
    \hline

 $n$ & $299$ & $300$ & $301$ & $302$ & $303$ & $304$ & $305$ & $306$ & $307$ & $308$\\
  $b(n)$ & $1$ & $152$ & $2$ & $2$ & $1$ & $1670$ & $2$ & 16$$ & $1$ & $23$\\ 
  $s(n)$ & $1$ & $8222$ & $16$ & $6$ & $1$ & $65466$ & $12$ & $294$ & $1$ & $311$\\
    \hline
    
   $n$ & $309$ & $310$ & $311$ & $312$ & $313$ & $314$ & $315$ & $316$ & $317$ & $318$\\
  $b(n)$ & $2$ & $6$ & $1$ & $507$ & $1$ & $2$ & $11$ & $9$ & $1$ & $4$\\ 
  $s(n)$ & $8$ & $94$ & $1$ & $32075$ & $1$ & $6$ & $47$ & $29$ & $1$ & $36$\\
   \hline
$n$ & $319$ & $320$ & $321$ & $322$ & $323$ & $324$ & $325$ & $326$ & $327$ & $328$\\ 
  $b(n)$ & $1$ & $?$ & $1$ & $4$ & $1$ & $10225$ & $4$ & $2$ & $2$ & $108$\\ 
  $s(n)$ & $1$ & $?$ & $1$ & $36$ & $1$ & $?$ & $4$ & $6$ & $8$ & $986$\\
   \hline
$n$ & $329$ & $330$ & $331$ & $332$ & $333$ & $334$ & $335$ & $336$ & $337$ & $338$\\
  $b(n)$ & $1$ & $12$ & $1$ & $9$ & $14$ & $2$ & $1$ & $10990$ & $1$ & $8$\\ 
  $s(n)$ & $1$ & $564$ & $1$ & $29$ & $80$ & $6$ & $1$ & $5247711$ & $1$ & $59$\\
   \hline
$n$ & $339$ & $340$ & $341$ & $342$ & $343$ & $344$ & $345$ & $346$ & $347$ & $348$\\
  $b(n)$ & $1$ & $35$ & $1$ & $42$ & $61$ & $90$ & $1$ & $2$ & $1$ & $28$\\ 
  $s(n)$ & $1$ & $739$ & $1$ & $1164$ & $373$ & $800$ & $1$ & $6$ & $1$ & $410$\\
  \hline
 
\end{tabular}
\hspace{1cm}
\caption{Further Computations}\label{Table1}
\end{small}
\end{table}

\begin{table}[H]
\centering
\begin{small}
 \begin{tabular}{|c | c c c c c c c c c c|}  
   \hline
$n$ & $349$ & $350$ & $351$ & $352$ & $353$ & $354$ & $355$ & $356$ & $357$ & $358$\\ 
  $b(n)$ & $1$ & $16$ & $166$ & $195479$ & $1$ & $4$ & $2$ & $11$ & $2$ & $2$\\ 
  $s(n)$ & $1$ & $306$ & $4591$ & $?$ & $1$ & $36$ & $12$ & $43$ & $8$ & $6$\\
   \hline
$n$ & $359$ & $360$ & $361$ & $362$ & $363$ & $364$ & $365$ & $366$ & $367$ & $368$\\ 
  $b(n)$ & $1$ & $2035$ & $4$ & $2$ & $5$ & $27$ & $1$ & $6$ & $1$ & $1670$\\ 
  $s(n)$ & $1$ & $535713$ & $4$ & $6$ & $20$ & $395$ & $1$ & $78$ & $1$ & $65466$\\
   \hline
$n$ & $369$ & $370$ & $371$ & $372$ & $373$ & $374$ & $375$ & $376$ & $377$ & $378$\\
  $b(n)$ & $4$ & $4$ & $1$ & $34$ & $1$ & $4$ & $54$ & $90$ & $1$ & $548$\\ 
  $s(n)$ & $4$ & $36$ & $1$ & $606$ & $1$ & $36$ & $253$ & $800$ & $1$ & $47244$\\
   \hline
$n$ & $379$ & $380$ & $381$ & $382$ & $383$ & $384$ & $385$ & $386$ & $387$ & $388$\\
  $b(n)$ & $1$ & $27$ & $2$ & $2$ & $1$ & $?$ & $2$ & $2$ & $11$ & $11$\\ 
  $s(n)$ & $1$ & $395$ & $8$ & $6$ & $1$ & $?$ & $12$ & $6$ & $47$ & $43$\\
   \hline
$n$ & $389$ & $390$ & $391$ & $392$ & $393$ & $394$ & $395$ & $396$ & $397$ & $398$\\
  $b(n)$ & $1$ & $12$ & $1$ & $463$ & $1$ & $2$ & $1$ & $111$ & $1$ & $2$\\ 
  $s(n)$ & $1$ & $468$ & $1$ & $18078$ & $1$ & $6$ & $1$ & $4985$ & $1$ & $6$\\
   \hline
$n$ & $399$ & $400$ & $401$ & $402$ & $403$ & $404$ & $405$ & $406$ & $407$ & $408$\\
  $b(n)$ & $5$ & $12744$ & $1$ & $6$ & $1$ & $11$ & $805$ & $6$ & $1$ & $399$\\ 
  $s(n)$ & $113$ & $3618636$ & $1$ & $78$ & $1$ & $43$ & $8453$ & $110$ & $1$ & $22923$\\

 \hline
 $n$ & $409$ & $410$ & $411$ & $412$ & $413$ & $414$ & $415$ & $416$ & $417$ & $418$\\
  $b(n)$ & $1$ & $6$ & $1$ & $9$ & $1$ & $16$ & $1$ & $209507$ & $2$ & $4$\\ 
  $s(n)$ & $1$ & $94$ & $1$ & $29$ & $1$ & $294$ & $1$ & $?$ & $8$ & $36$\\

 \hline
  $n$ & $419$ & $420$ & $421$ & $422$ & $423$ & $424$ & $425$ & $426$ & $427$ & $428$\\ 
  $b(n)$ & $1$ & $104$ & $1$ & $2$ & $4$ & $106$ & $4$ & $4$ & $1$ & $9$\\ 
  $s(n)$ & $1$ & $9052$ & $1$ & $6$ & $4$ & $944$ & $4$ & $36$ & $1$ & $29$\\

 \hline
 $n$ & $429$ & $430$ & $431$ & $432$ & $433$ & $434$ & $435$ & $436$ & $437$ & $438$\\ 
  $b(n)$ & $2$ & $4$ & $1$ & $115708$ &  $1$ & $4$  & $1$ & $11$ & $1$  & $6$\\ 
  $s(n)$ & $8$ & $36$ & $1$ & $?$ &  $1$ & $36$ &$1$ & $43$ & $1$ & $78$\\

 \hline
  $n$ & $439$ & $440$ & $441$ & $442$ & $443$ & $444$ & $445$ & $446$ & $447$ & $448$\\
  $b(n)$ & $1$ & $474$ & $55$ & $4$ & $1$ & $40$ & $1$ & $2$ & $1$ & $?$\\ 
  $s(n)$ & $1$ & $31970$ & $1110$ & $36$ & $1$ & $782$ & $1$ & $6$ & $1$ & $?$\\
  
   \hline
 $n$ & $449$ & $450$ & $451$ & $452$ & $453$ & $454$ & $455$ & $456$ & $457$ & $458$\\
  $b(n)$ & $1$ & $82$ & $1$ & $11$ & $2$ & $2$ & $1$ & $441$ & $1$ & $2$\\ 
  $s(n)$ & $1$ & $3797$ & $1$ & $43$ & $8$ & $6$ & $1$ & $28447$ & $1$ & $6$\\

\hline
 $n$ & $459$ & $460$ & $461$ & $462$ & $463$ & $464$ & $465$ & $466$ & $467$ & $468$\\
  $b(n)$ & $37$ & $27$ & $1$ & $12$ & $1$ & $1984$ & $4$ & $2$ & $1$ & $267$\\ 
  $s(n)$ & $101$ & $395$ & $1$ & $468$ & $1$ & $74104$ & $66$ & $6$ & $1$ & $13941$\\
  
  \hline
 $n$ & $469$ & $470$ & $471$ & $472$ & $473$ & $474$ & $475$ & $476$ & $477$ & $478$\\
  $b(n)$ & $1$ & $4$ & $2$ & $90$ & $1$ & $6$ & $4$ & $27$ & $4$ & $2$\\ 
  $s(n)$ & $1$ & $36$ & $8$ & $800$ & $1$ & $78$ & $4$ & $395$ & $4$ & $6$\\
  
  \hline
 $n$ & $479$ & $480$ & $481$ & $482$ & $483$ & $484$ & $485$ & $486$ & $487$ & $488$\\
  $b(n)$ & $1$ & $?$ & $1$ & $2$ & $2$ & $41$ & $1$ & $639775$ & $1$ & $106$\\ 
  $s(n)$ & $1$ & $?$ & $1$ & $6$ & $8$ & $421$ & $1$ & $?$ & $1$ & $944$\\
  
  \hline
 $n$ & $489$ & $490$ & $491$ & $492$ & $493$ & $494$ & $495$ & $496$ & $497$ & $498$\\
  $b(n)$ & $2$ & $16$ & $1$ & $28$ & $1$ & $4$ & $8$ & $1670$ & $2$ & $4$\\ 
  $s(n)$ & $8$ & $318$ & $1$ & $410$ & $1$ & $36$ & $48$ & $65466$ & $16$ & $36$\\
  
  \hline
 $n$ & $499$ & $500$ & $501$ & $502$ & $503$ & $504$ & $505$ & $506$ & $507$ & $508$\\
  $b(n)$ & $1$ & $634$ & $1$ & $2$ & $1$ & $3249$ & $2$ & $6$ & $9$ & $9$\\ 
  $s(n)$ & $1$ & $21252$ & $1$ & $6$ & $1$ & $871013$ & $12$ & $142$ & $135$ & $29$\\
  
  \hline
 $n$ & $509$ & $510$ & $511$ & $512$ & $513$ & $514$ & $515$ & $516$ & $517$ & $518$\\
  $b(n)$ & $1$ & $8$ & $1$ & $?$ & $189$ & $2$ & $1$ & $34$ & $1$ & $4$\\ 
  $s(n)$ & $1$ & $216$ & $1$ & $?$ & $5055$ & $6$ & $1$ & $606$ & $1$ & $36$\\

 \hline
  $n$ & $519$ & $520$ & $521$ & $522$ & $523$ & $524$ & $525$ & $526$ & $527$ & $528$\\
  $b(n)$ & $1$ & $484$ & $1$ & $16$ & $1$ & $9$ & $10$ & $2$ & $1$ & $9274$\\ 
  $s(n)$ & $1$ & $28714$ & $1$ & $294$ & $1$ & $29$ & $112$ & $6$ & $1$ & $4381956$\\ 

 \hline
\end{tabular}
\hspace{1cm}
\caption{Further Computations}\label{Table1}
\end{small}
\end{table}

\begin{table}[H]
\centering
\begin{small}
 \begin{tabular}{|c | c c c c c c c c c c|}

   \hline
  $n$ & $529$ & $530$ & $531$ & $532$ & $533$ & $534$ & $535$ & $536$ & $537$ & $538$\\ 
  $b(n)$ & $4$ & $4$ & $4$ & $23$ & $1$ & $4$ & $1$ & $90$ & $1$ & $2$\\ 
  $s(n)$ & $4$ & $36$ & $4$ & $311$ & $1$ & $36$ & $1$ & $800$ & $1$ & $6$\\ 
  
   \hline
  $n$ & $539$ & $540$ & $541$ & $542$ & $543$ & $544$ & $545$ & $546$ & $547$ & $548$\\
  $b(n)$ & $4$ & $1342$ & $1$ & $2$ & $2$ & $210043$ & $1$ & $24$ & $1$ & $11$\\ 
  $s(n)$ & $4$ & $148151$ & $1$ & $6$ & $8$ & $?$ & $1$ & $2664$ & $1$ & $43$\\ 
  
   \hline
  $n$ & $549$ & $550$ & $551$ & $552$ & $553$ & $554$ & $555$ & $556$ & $557$ & $558$\\
  $b(n)$ & $11$ & $40$ & $1$ & $345$ & $1$ & $2$ & $2$ & $9$ & $1$ & $36$\\ 
  $s(n)$ & $47$ & $1370$ & $1$ & $20231$ & $1$ & $6$ & $8$ & $29$ & $1$ & $990$\\ 
  
   \hline
  $n$ & $559$ & $560$ & $561$ & $562$ & $563$ & $564$ & $565$ & $566$ & $567$ & $568$\\
  $b(n)$ & $1$ & $10423$ & $1$ & $2$ & $1$ & $24$ & $1$ & $2$ & $7196$ & $90$\\ 
  $s(n)$ & $1$ & $4633376$ & $1$ & $6$ & $1$ & $324$ & $1$ & $6$ & $2253564$ & $800$\\ 
  
   \hline
  $n$ & $569$ & $570$ & $571$ & $572$ & $573$ & $574$ & $575$ & $576$ & $577$ & $578$\\
  $b(n)$ & $1$ & $12$ & $1$ & $27$ & $1$ & $4$ & $4$ & $?$ & $1$ & $8$\\ 
  $s(n)$ & $1$ & $468$ & $1$ & $395$ & $1$ & $36$ & $4$ & $?$ & $1$ & $63$\\ 
  
   \hline
  $n$ & $579$ & $580$ & $581$ & $582$ & $583$ & $584$ & $585$ & $586$ & $587$ & $588$\\
  $b(n)$ & $2$ & $35$ & $1$ & $6$ & $1$ & $108$ & $11$ & $2$ & $1$ & $202$\\ 
  $s(n)$ & $8$ & $739$ & $1$ & $78$ & $1$ & $986$ & $47$ & $6$ & $1$ & $21836$\\ 
  
   \hline
  $n$ & $589$ & $590$ & $591$ & $592$ & $593$ & $594$ & $595$ & $596$ & $597$ & $598$\\
  $b(n)$ & $1$ & $4$ & $1$ & $1984$ & $1$ & $160$ & $1$ & $11$ & $2$ & $4$\\ 
  $s(n)$ & $1$ & $36$ & $1$ & $74104$ & $1$ & $6168$ & $1$ & $43$ & $8$ & $36$\\ 

 \hline
  $n$ & $599$ & $600$ & $601$ & $602$ & $603$ & $604$ & $605$ & $606$ & $607$ & $608$\\
  $b(n)$ & $1$ & $2413$ & $1$ & $6$ & $11$ & $9$ & $10$ & $4$ & $1$ & $195479$\\ 
  $s(n)$ & $1$ & $659897$ & $1$ & $110$ & $47$ & $29$ & $409$ & $36$ & $1$ & $?$\\ 
  
   \hline
  $n$ & $609$ & $610$ & $611$ & $612$ & $613$ & $614$ & $615$ & $616$ & $617$ & $618$\\
  $b(n)$ & $3$ & $6$ & $1$ & $129$ & $1$ & $2$ & $2$ & $335$ & $1$ & $6$\\ 
  $s(n)$ & $25$ & $94$ & $1$ & $5835$ & $1$ & $6$ & $12$ & $19885$ & $1$ & $78$\\ 
  
   \hline
  $n$ & $619$ & $620$ & $621$ & $622$ & $623$ & $624$ & $625$ & $626$ & $627$ & $628$\\
  $b(n)$ & $1$ & $36$ & $37$ & $2$ & $1$ & $12547$ & $2308$ & $2$ & $2$ & $11$\\ 
  $s(n)$ & $1$ & $702$ & $101$ & $6$ & $1$ & $5595183$ & $69032$ & $6$ & $8$ & $43$\\ 
  
   \hline
  $n$ & $629$ & $630$ & $631$ & $632$ & $633$ & $634$ & $635$ & $636$ & $637$ & $638$\\
  $b(n)$ & $1$ & $72$ & $1$ & $90$ & $2$ & $2$ & $1$ & $28$ & $4$ & $4$\\ 
  $s(n)$ & $1$ & $5940$ & $1$ & $800$ & $8$ & $6$ & $1$ & $410$ & $4$ & $36$\\ 
  
   \hline
  $n$ & $639$ & $640$ & $641$ & $642$ & $643$ & $644$ & $645$ & $646$ & $647$ & $648$\\
  $b(n)$ & $4$ & $?$ & $1$ & $4$ & $1$ & $23$ & $2$ & $4$ & $1$ & $91071$\\ 
  $s(n)$ & $4$ & $?$ & $1$ & $36$ & $1$ & $311$ & $8$ & $36$ & $1$ & $?$\\ 
  
   \hline
  $n$ & $649$ & $650$ & $651$ & $652$ & $653$ & $654$ & $655$ & $656$ & $657$ & $658$\\
  $b(n)$ & $1$ & $16$ & $5$ & $9$ & $1$ & $6$ & $2$ & $2010$ & $14$ & $4$\\ 
  $s(n)$ & $1$ & $306$ & $113$ & $29$ & $1$ & $78$ & $12$ & $74860$ & $80$ & $36$\\ 
  
   \hline
  $n$ & $659$ & $660$ & $661$ & $662$ & $663$ & $664$ & $665$ & $666$ & $667$ & $668$\\
  $b(n)$ & $1$ & $100$ & $1$ & $2$ & $2$ & $90$ & $1$ & $42$ & $1$ & $9$\\ 
  $s(n)$ & $1$ & $9346$ & $1$ & $6$ & $8$ & $800$ & $1$ & $1164$ & $1$ & $29$\\ 
  
   \hline
  $n$ & $669$ & $670$ & $671$ & $672$ & $673$ & $674$ & $675$ & $676$ & $677$ & $678$\\
  $b(n)$ & $2$ & $4$ & $1$ & $?$ & $1$ & $2$ & $232$ & $51$ & $1$ & $4$\\ 
  $s(n)$ & $8$ & $36$ & $1$ & $?$ & $1$ & $6$ & $3682$ & $791$ & $1$ & $36$\\ 
  
   \hline
  $n$ & $679$ & $680$ & $681$ & $682$ & $683$ & $684$ & $685$ & $686$ & $687$ & $688$\\
  $b(n)$ & $1$ & $492$ & $1$ & $4$ & $1$ & $259$ & $1$ & $128$ & $2$ & $1670$\\ 
  $s(n)$ & $1$ & $29698$ & $1$ & $36$ & $1$ & $12723$ & $1$ & $2084$ & $8$ & $65466$\\ 
  
   \hline
  $n$ & $689$ & $690$ & $691$ & $692$ & $693$ & $694$ & $695$ & $696$ & $697$ & $698$\\
  $b(n)$ & $2$ & $8$ & $1$ & $11$ & $11$ & $2$ & $1$ & $395$ & $1$ & $2$\\ 
  $s(n)$ & $28$ & $216$ & $1$ & $43$ & $47$ & $6$ & $1$ & $22667$ & $1$ & $6$\\ 
  
   \hline
  $n$ & $699$ & $700$ & $701$ & $702$ & $7703$ & $704$ & $705$ & $706$ & $707$ & $708$\\
  $b(n)$ & $1$ & $126$ & $1$ & $550$ & $1$ & $?$ & $1$ & $2$ & $1$ & $24$\\ 
  $s(n)$ & $1$ & $7102$ & $1$ & $47374$ & $1$ & $?$ & $1$ & $6$ & $1$ & $324$\\

  \hline
\end{tabular}
\hspace{1cm}
\caption{Further Computations}\label{Table1}
\end{small}
\end{table}

\begin{table}[H]
\centering
\begin{small}
 \begin{tabular}{|c | c c c c c c c c c c|} 
 \hline
  $n$ & $709$ & $710$ & $711$ & $712$ & $713$ & $714$ & $715$ & $716$ & $717$ & $718$\\
  $b(n)$ & $1$ & $6$ & $11$ & $108$ & $1$ & $12$ & $2$ & $9$ & $1$ & $2$\\ 
  $s(n)$ & $1$ & $94$ & $47$ & $986$ & $1$ & $468$ & $12$ & $29$ & $1$ & $6$\\ 
  
   \hline
  $n$ & $719$ & $720$ & $721$ & $722$ & $723$ & $724$ & $725$ & $726$ & $727$ & $728$\\
  $b(n)$ & $1$ & $65074$ & $1$ & $8$ & $2$ & $11$ & $4$ & $19$ & $1$ & $385$\\ 
  $s(n)$ & $1$ & $?$ & $1$ & $65$ & $8$ & $43$ & $4$ & $466$ & $1$ & $22295$\\  

\hline
  $n$ & $729$ & $730$ & $731$ & $732$ & $733$ & $734$ & $735$ & $736$ & $737$ & $738$\\
  $b(n)$ & $?$ & $4$ & $1$ & $40$ & $1$ & $2$ & $9$ & $195479$ & $2$ & $16$\\ 
  $s(n)$ & $?$ & $36$ & $1$ & $782$ & $1$ & $6$ & $123$ & $?$ & $24$ & $294$\\ 
  
   \hline
  $n$ & $739$ & $740$ & $741$ & $742$ & $743$ & $744$ & $745$ & $746$ & $747$ & $748$\\
  $b(n)$ & $1$ & $35$ & $5$ & $4$ & $1$ & $441$ & $1$ & $2$ & $4$ & $27$\\ 
  $s(n)$ & $1$ & $739$ & $113$ & $36$ & $1$ & $28447$ & $1$ & $6$ & $4$ & $395$\\ 

 \hline
  $n$ & $749$ & $750$ & $751$ & $752$ & $753$ & $754$ & $755$ & $756$ & $757$ & $758$\\
  $b(n)$ & $1$ & $224$ & $1$ & $1670$ & $1$ & $4$ & $2$ & $3757$ & $1$ & $2$\\ 
  $s(n)$ & $1$ & $10001$ & $1$ & $65466$ & $1$ & $36$ & $12$ & $794193$ & $1$ & $6$\\ 

 \hline
  $n$ & $759$ & $760$ & $761$ & $762$ & $763$ & $764$ & $765$ & $766$ & $767$ & $768$\\
  $b(n)$ & $2$ & $384$ & $1$ & $6$ & $1$ & $9$ & $4$ & $2$ & $1$ & $?$\\ 
  $s(n)$ & $24$ & $22278$ & $1$ & $78$ & $1$ & $29$ & $4$ & $6$ & $1$ & $?$\\ 

 \hline
  $n$ & $769$ & $770$ & $771$ & $772$ & $773$ & $774$ & $775$ & $776$ & $777$ & $778$\\
  $b(n)$ & $1$ & $12$ & $1$ & $11$ & $1$ & $36$ & $13$ & $108$ & $5$ & $2$\\ 
  $s(n)$ & $1$ & $564$ & $1$ & $43$ & $1$ & $990$ & $93$ & $986$ & $113$ & $6$\\ 

 \hline
  $n$ & $779$ & $780$ & $781$ & $782$ & $783$ & $784$ & $785$ & $786$ & $787$ & $788$\\ 
  $b(n)$ & $1$ & $128$ & $1$ & $4$ & $37$ & $9998$ & $1$ & $4$ & $1$ & $11$\\ 
  $s(n)$ & $1$ & $13320$ & $1$ & $36$ & $101$ & $3074483$ & $1$ & $36$ & $1$ & $43$\\ 

 \hline
  $n$ & $789$ & $790$ & $791$ & $792$ & $793$ & $794$ & $795$ & $796$ & $797$ & $798$\\
  $b(n)$ & $1$ & $4$ & $2$ & $1771$ & $1$ & $2$ & $1$ & $9$ & $1$ & $24$\\ 
  $s(n)$ & $1$ & $36$ & $16$ & $484183$ & $1$ & $6$ & $1$ & $29$ & $1$ & $2664$\\ 

 \hline
  $n$ & $799$ & $800$ & $801$ & $802$ & $803$ & $804$ & $805$ & $806$ & $807$ & $808$\\
  $b(n)$ & $1$ & $?$ & $4$ & $2$ & $1$ & $34$ & $1$ & $4$ & $1$ & $106$\\ 
  $s(n)$ & $1$ & $?$ & $4$ & $6$ & $1$ & $606$ & $1$ & $36$ & $1$ & $944$\\ 

 \hline
  $n$ & $809$ & $810$ & $811$ & $812$ & $813$ & $814$ & $815$ & $816$ & $817$ & $818$\\
  $b(n)$ & $1$ & $2751$ & $1$ & $38$ & $2$ & $4$ & $1$ & $10604$ & $1$ & $2$\\ 
  $s(n)$ & $1$ & $272960$ & $1$ & $920$ & $8$ & $36$ & $1$ & $4658179$ & $1$ & $6$\\ 

 \hline
  $n$ & $819$ & $820$ & $821$ & $822$ & $823$ & $824$ & $825$ & $826$ & $827$ & $828$\\
  $b(n)$ & $41$ & $46$ & $1$ & $4$ & $1$ & $90$ & $14$ & $4$ & $1$ & $111$\\ 
  $s(n)$ & $1337$ & $1212$ & $1$ & $36$ & $1$ & $800$ & $105$ & $36$ & $1$ & $4985$\\ 
  
 \hline
  $n$ & $829$ & $830$ & $831$ & $832$ & $833$ & $834$ & $835$ & $836$ & $837$ & $838$\\
  $b(n)$ & $1$ & $4$ & $2$ & $?$ & $4$ & $6$ & $1$ & $23$ & $165$ & $2$\\ 
  $s(n)$ & $1$ & $36$ & $8$ & $?$ & $4$ & $78$ & $1$ & $311$ & $4560$ & $6$\\ 

 \hline
  $n$ & $839$ & $840$ & $841$ & $842$ & $843$ & $844$ & $845$ & $846$ & $847$ & $848$\\
  $b(n)$ & $1$ & $1933$ & $4$ & $2$ & $1$ & $9$ & $4$ & $16$ & $4$ & $1984$ \\ 
  $s(n)$ & $1$ & $878779$ & $4$ & $6$ & $1$ & $27$ & $4$ & $294$ & $4$ & $74104$ \\ 

 \hline
  $n$ & $849$ & $850$ & $851$ & $852$ & $853$ & $854$ & $855$ & $856$ & $857$ & $858$\\
  $b(n)$ & $2$ & $16$ & $1$ & $24$ & $1$ & $4$ & $14$ & $90$ & $1$ & $12$\\ 
  $s(n)$ & $8$ & $306$ & $1$ & $324$ & $1$ & $36$ & $80$ & $800$ & $1$ & $468$\\ 

 \hline
  $n$ & $859$ & $860$ & $861$ & $862$ & $863$ & $864$ & $865$ & $866$ & $867$ & $868$\\
  $b(n)$ & $1$ & $27$ & $2$ & $2$ & $1$ & $?$ & $1$ & $2$ & $5$ & $23$\\ 
  $s(n)$ & $1$ & $395$ & $8$ & $6$ & $1$ & $?$ & $1$ & $6$ & $26$ & $311$\\ 

 \hline
\end{tabular}
\hspace{1cm}
\caption{Further Computations}\label{Table1}
\end{small}
\end{table}

We now record some partial computations  considering specific additive groups of given orders.

\begin{table}[H]
\centering
\begin{small}
 \begin{tabular}{| c | c | c | c | c | c | c | c | c | c | c | c|} 
 \hline
 $Group \; Id$ & $[64,1]$ & $[64,2]$ & $[64,26]$ & $[64,50]$ & $[64,55]$ & $[64,83]$\\ 
  $Number$ & $10$ & $11354$ & $2742$ & $142$ & $?$ & $734410$\\ 
 \hline
 \hline 
  $Group \; Id$ & $[64,183]$ & $[64,192]$ & $[64,246]$ & $[64,260]$ & $[64,267]$ & \\
  $Number$ & $3124$ & $?$ & $253350$ & $2189661$ & $58558$ & \\
 \hline
\end{tabular}
\hspace{1cm}
\caption{Enumerations of left braces of order 64 }\label{Table1}
\end{small}
\end{table}

\begin{table}[H]
\centering
\begin{small}
 \begin{tabular}{| c | c | c | c | c | c | c | c | c | c | c | c|} 
 \hline
 $Group \; Id$ & $[480,4]$ & $[480,199]$ & $[480,212]$ & $[480,919]$ & $[480,934]$ & $[480,1180]$ & $[480,1213]$\\ 
  $Number$ & $128$ & $?$ & $4928$ & $958965$ & $99970$ & $?$ & $39650$\\ 
 \hline
\end{tabular}
\hspace{1cm}
\caption{Enumerations of left braces of order 480}\label{Table1}
\end{small}
\end{table}

\section{Conclusion and Conjectures}

We start by presenting a comparison on the time taken ( in seconds) by \cite[Algorithm 5.1]{GV2017} and Algorithm \ref{alg1} for enumerating skew left braces of  order 32 for select additive groups  which took considerable amount of time on MAGMA.

\begin{table}[H]
\centering
\begin{small}
 \begin{tabular}{| c | c | c | c | c | c | c | c | c | c | c | c|} 
 \hline
$Group \; Id \; of \;the \;additive \;group$ & $[32,23]$ & $[32,24]$ & $[32,25]$ & $[32,28]$ & $[32,29]$ & $[32,30]$\\
  $Number \;of  \;skew \;brace\; structures$ & $39488$ & $70400$ &$138336$ & $138336$ & $138336$ & $137526$\\ 
  $Time\; on  \; Algorithm \, 5.1   \;[13]$ & $11238$ & $9808$ &$18720$ & $10193$ & $10083$ & $34005$\\ 
  $Time \; on  \; Algorithm \; \ref{alg1}$ & $539$ & $709$ & $1905$ & $4308$ & $3135$ & $4658$\\
   \hline
   \hline
   $Group \; Id  \; of \;the \;additive \;group$ & $[32,31]$ & $[32,32]$ & $[32,33]$ & $[32,45]$ & $[32,47]$ & $[32,51]$\\
   $Number \;of  \;skew \;brace\; structures$ & $70944$ & $69236$ & $91008$ & $8015$ & $7870$ & $744$\\
    $Time\; on  \; Algorithm \, 5.1   \;[13]$ & $14568$ & $18342$ & $17222$ & $130942$ & $28848$ &$\#$\\ 
    $Time \; on  \; Algorithm \; \ref{alg1}$ & $4797$ & $9302$ & $8869$ & $30$ & $68$ & $8$\\
    \hline
\end{tabular}
\hspace{1cm}
\caption{Time comparison on skew left  braces of order 32}\label{Table1}
\end{small}
\end{table}

$\#$ Program was stopped after running more than a month without result.

\vspace{.1in}

The data obtained above reveals that Algorithm \ref{alg1} is very expensive, with respect to memory space and time, for handing the situation for prime power orders. So, one really needs to find a substitute for this algorithm. One may think of  Algorithm \ref{alg2} as a substitute. But unfortunately, it requires  the conjugacy classes of regular subgroups of a given Sylow-$p$-subgroup (of the holomorph of a given finite $p$-group) to be computed  in the whole holomorph, which is again very expensive. Although Algorithm \ref{alg2} is not very efficient as such, we hope that it may be improved/modified to handle the computations on skew braces  of prime power orders more efficiently.

We now present some conjectures suggested by the big data computed in above  tables. It is known from \cite{D2020} that for a prime integer $q \ge 5$,

$$
b(4q) =
\begin{cases}
9,   &\text{if} \;\; q \equiv 3 \mod{4}\\
11, &\text{if } \;q \equiv 1 \mod{4}
\end{cases}
$$

and for prime integers $p$ and $q$ such that $q > p+1 > 3$,

$$
b(p^2q) =
\begin{cases}
4,   &\text{if} \;\; 3 \nmid p-1\\
p+8,   &\text{if} \;\; 3 \mid p-1 \;\ \text{and}\;\ 9 \nmid p-1 \\
2p+8, &\text{if } \; 9 \mid p-1.
\end{cases}
$$

For  skew left braces, we have

\begin{conj} Let $p$ and $q$ be  prime integers. If $q \ge 5$, then
$$
s(4q) =
\begin{cases}
29,   &\text{if} \;\; q \equiv 3 \mod{4}\\
43, &\text{if } \; q \equiv 1 \mod{4}
\end{cases}
$$

and if $q > p+1 > 3$, then

$$
s(p^2q) =
\begin{cases}
4,   &\text{if} \;\; p \nmid q-1\\
2p^2+7p+8,   &\text{if} \;\; p \mid q-1 \;\ \text{and}\;\ p^2 \nmid q-1 \\
6p^2+6p+8, &\text{if } \; p^2 \mid q-1.
\end{cases}
$$
\end{conj}

For prime multiples of  $8$ and $12$, we have

\begin{conj} Let $p \ge 11$ be a prime integer. Then
$$
b(8p) =
\begin{cases}
90,   &\text{if} \;\; p \equiv 3,\;7 \mod{8}\\
106,   &\text{if} \;\; p \equiv 5 \mod{8}\\
108, &\text{if } \; p \equiv 1 \mod{8}.
\end{cases}
$$

and

$$
s(8p) =
\begin{cases}
800,   &\text{if} \;\; p \equiv 3,\;7 \mod{8}\\
944,   &\text{if} \;\; p \equiv 5 \mod{8}\\
986, &\text{if } \; p \equiv 1 \mod{8}.
\end{cases}
$$
\end{conj}

\begin{conj} Let $p \ge 7$ be a prime integer. Then
$$
b(12p) =
\begin{cases}
24,   &\text{if} \;\; p \equiv 11 \mod{12}\\
28,   &\text{if} \;\; p \equiv 5 \mod{12}\\
34,   &\text{if} \;\; p \equiv 7 \mod{12}\\
40, &\text{if } \; p \equiv 1 \mod{12}.
\end{cases}
$$

and

$$
s(12p) =
\begin{cases}
324,   &\text{if} \;\; p \equiv 11 \mod{12}\\
410,   &\text{if} \;\; p \equiv 5 \mod{12}\\
606,   &\text{if} \;\; p \equiv 7 \mod{12}\\
782, &\text{if } \; p \equiv 1 \mod{12}.
\end{cases}
$$
\end{conj}

Skew left braces of order $pq$, $p < q$ being prime integers, have been constructed very recently in \cite{AB20}, where it is shown that $s(pq) = 1$ if $p \nmid q-1$ and $s(pq) = 2p+2$ otherwise. Going a step ahead, we have the following enumeration formula:

\begin{conj} Let $p$ and $q$ be  prime integers such that $q > p \ge 3$. Then
$$
b(2pq) =
\begin{cases}
4,   &\text{if} \;\; p \nmid q-1\\
6, &\text{if } \; p \mid q-1
\end{cases}
$$

and 

$$
s(2pq) =
\begin{cases}
36,   &\text{if} \;\; p \nmid q-1\\
8p+54,   &\text{if} \;\; p \mid q-1.
\end{cases}
$$
\end{conj}

We close with the hope that the readers will be able to use the enormous data produced above to formulate many more conjectures according to their own need and interest.
\vspace{.2in}

\noindent{\it Acknowledgements.}  The third named author thanks L. Vendramin for supplying MAGMA codes for computing skew left braces and for his useful comments on the introduction, and  acknowledges the support of DST-RSF Grant INT/RUS/RSF/P-2. The first and second named authors acknowledge the support from the RFBR-18-01-0057. The authors thank the referee for suggesting useful modifications.

\end{document}